\documentclass[11]{amsart}
\usepackage{amsmath}
\usepackage{amsfonts}
\usepackage{amssymb}
\usepackage{latexsym}
\usepackage{amsthm}
\makeatletter
\@namedef{subjclassname@2010}{%
  \textup{2010} Mathematics Subject Classification}
\makeatother

\usepackage{hyperref}
\usepackage{geometry}

\usepackage{xypic}

\newtheorem*{thmm}{Main Theorem}

\newtheorem*{hypothesis a}{Hypothesis A}

\everymath{\displaystyle}

\newtheorem{thm}{Theorem}[section]

\newtheorem{lem}[thm]{Lemma}
\newtheorem{prop}[thm]{Proposition}

\theoremstyle{definition}
\newtheorem{df}[thm]{Definition}
\newtheorem{example}[thm]{Example}

\numberwithin{equation}{section}

\newcommand{\Ad}{\textrm{ad}}
\newcommand{\D}{\textrm{dim}}

\newcommand{\Frob}{\textrm{Frob}}
\newcommand{\rb}{\overline\rho}
\newcommand{\Q}{\mathbb{Q}}
\newcommand{\Z}{\mathbb{Z}}
\newcommand{\F}{\mathbb{F}}
\newcommand{\e}{\epsilon}

\newcommand{\CD}{\mathcal{D}}

\newcommand{\CO}{\mathcal{O}}
\newcommand{\fp}{\mathfrak{p}}
\newcommand{\fq}{\mathfrak{q}}
\newcommand{\fn}{\mathfrak{n}}
\newcommand{\fr}{\mathfrak{r}}

\newcommand{\fm}[4]{\begin{pmatrix} #1 & #2 \\ #3 & #4 \end{pmatrix}}

\newcommand{\sfm}[4]{\left(\begin{smallmatrix} #1 & #2 \\ #3 & #4 \end{smallmatrix}
\right)}

\begin{document}
\renewcommand*\arraystretch{2}

 \title{Higher congruence companion forms}

 \title{Higher congruence companion forms}
\author{Rajender Adibhatla}
\address{ Fakult\"{a}t f\"{u}r Mathematik, Universit\"{a}t Regensburg,
93040 Regensburg, Germany }
                     \email{rajender.adibhatla@mathematik.uni-regensburg.de}

\author{Jayanta Manoharmayum}
\address{School of Mathematics and Statistics, University of Sheffield,
                          Sheffield S3 7RH, United Kingdom}
                          \email{j.manoharmayum@sheffield.ac.uk}


\begin{abstract}
For a rational prime $p \geq 3$ we consider $p$-ordinary,
Hilbert modular newforms $f$ of weight $k\geq 2$ with associated $p$-adic Galois
representations $\rho_f$ and mod ${p^n}$ reductions
$\rho_{f,n}$.
Under suitable hypotheses on the size of the image, we use deformation theory and
modularity lifting to show
that if the restrictions of  $\rho_{f,n} $ to
decomposition groups above $p$ split then $f$ has a companion form
$g$ modulo $p^n$ (in the sense that $\rho_{f,n}\sim \rho_{g,n}\otimes\chi^{k-1}$).

 \end{abstract}



\maketitle

\section{Introduction}

Let $F$ be a totally real number field and let $p$ be an odd prime. Suppose we are
given a Hilbert modular newform $f$ over $F$ of level $\fn_f$, character
$\psi$  and (parallel) weight $k\geq 2$. For a prime $\fq$ not dividing
$\fn_f$, let $c(\fq,f)$ denote the eigenvalue of
the Hecke operator $T(\mathfrak{q})$ acting on $f$;
denote by $K_f$ the number field generated by the $c(\fq,f)$'s and $\psi(\Frob_
{\mathfrak{q}})$'s, and by $\mathcal{O}_f $ the integer ring of $K_f$. Then for each
prime
$\wp|p$ of $\mathcal{O}_f$  one has a continuous, odd, absolutely irreducible
representation
$\rho_{f,\wp}: G_F\longrightarrow GL_2(\mathcal{O}_{f,\wp})$ with the
following property: $\rho_{f,\wp}$ is  unramified
outside primes dividing $p\mathfrak{n}_f$, and at a prime
$\mathfrak{q}\nmid p\mathfrak{n}_f$ the characteristic polynomial of
$\rho_{f,\wp}( \Frob_{\mathfrak{q}} ) $ is
$X^2-c(\mathfrak{q},f)X+\psi(\Frob_{\mathfrak{q}})\text{Nm}(\mathfrak{q})^{k-1}$. We
denote the $p$-adic cyclotomic character by $\chi $. Thus the determinant of
$\rho_{f,\wp}$ is $\psi\chi^{k-1}$.

From here on we assume that  the character $\psi$ is unramified at $p$. Suppose that $f
$  is ordinary at $p$. Then, by Wiles \cite{W1}, and Mazur-Wiles \cite{MW}, for
every prime $\mathfrak{p}|p$ we have
 \[
 \rho_{f,\wp}|_{G_\mathfrak{p}}\sim \fm{\psi_{1\mathfrak{p}}\chi^{k-1}}{*}{0}{\psi_
{2\mathfrak{p}}}
 \]
  where
 $G_\mathfrak{p}$ is a decomposition group at $\mathfrak{p}$ and $\psi_{1\mathfrak{p}}$, $\psi_{2\mathfrak{p}}$ are unramified characters. In fact,
 with $a(\mathfrak{p},f)$ defined to be the unit root
 $X^2-c(\mathfrak{p},f)X+\psi(\Frob_{\mathfrak{p}})\text{Nm}(\mathfrak{p})^{k-1}=0$, we
 have $\psi_{2\mathfrak{p}}(\Frob_
\mathfrak{p})=c(\mathfrak{p}, f)$.
  A natural question
is to ask when  the restriction(s) $\rho_{f,\wp}|_{G_\mathfrak{p}}$ actually split.
If $\rho_{f,\wp}$ mod $\wp$ is absolutely irreducible then
the splitting (or not) of $\rho_{f,\wp}|_{G_\mathfrak{p}}\mod{\wp^n}$ is independent of the  choice of a
lattice used to define $\rho_{f,\wp}$. This is explained in greater detail at the beginning of Section 3.

Now suppose we are given a second newform $g$ which is also  ordinary at $p$. Fix a $p
$-adic integer ring $\mathcal{O}$ in which
$\mathcal{O}_f$ and $\mathcal{O}_g$ embed. Given a non-zero, non-unit $\pi\in \CO$, we say that
 $g$ is a \emph{mod $\pi$ weak companion form} for $f$  if
 $c(\mathfrak{q},f)\equiv c(\mathfrak{q},g)\text{Nm}(\mathfrak{q})^{k-1}\mod{\pi}$
 for all but finitely many primes
$\mathfrak{q}$.  If the residual representations are absolutely irreducible---which will be the case in this article---then $g$ being a mod $\pi$ weak companion form to $f$ is equivalent to  $\rho_f\sim \rho_g\otimes\chi^{k-1}\mod\pi$ where $\rho_f,\rho_g:G_F\longrightarrow GL_2(\mathcal{O})$ are the  $p$-adic Galois representations associated to $f, g$. By applying the determinant condition to this companionship criterion, we see that the weight $k'$ of $g$ satisfies a  congruence
 $\chi^{k'-1}\equiv \chi^{1-k}\mod{\pi}$ on each decomposition group above $p$. Note
 that we do not enforce any optimality requirement on the level of $g$ (and hence the
 prefix  `weak').

Classically, companion forms mod $p$  played an important part in
the weight optimisation part of Serre's Modularity Conjecture. Serre's predicted
equivalence between local splitting for the residual modular representation (tame
ramification)
and the existence
of companions  was established by Gross in \cite{Gross}.
In much the same spirit, the main result of this paper, which we now state, proves the
equivalence between splitting mod ${p^n}$ and the existence of mod ${p^n}$ weak
companion forms.

\begin{thmm}\label{maint} Let $F$ be a totally real number field, $p$ be an odd
prime unramified in $F,$ and let $f$ be a $p$-ordinary Hilbert modular newform $f$
of  squarefree level $\mathfrak{n}$,  character $\psi$ with order
coprime to $p$ and
unramified at $p$, and weight $k\geq 2$.
Let $n\geq 2$ and  set  $\pmb{k}:=\mathcal{O}_f/\wp$,  $\rho_{f,n}:=\rho_{f,\wp}\mod{p^n}$, $\rb_f:=
\rho_{f,\wp}\mod{\wp}$. Assume the following
hypotheses:
\begin{itemize}
\item Global conditions.
\begin{enumerate}
\item[(GC1)] $\rho_{f,n}$ takes values in
$GL_2(W/p^n)$ where $W$ is the Witt ring of $\pmb{k}:=\mathcal{O}_f/\wp$  (under the
natural injection  $W/p^n\hookrightarrow \mathcal{O}_f/p^n$).
\item[(GC2)] The image of $\rb_f$ contains $SL_2(\pmb{k})$.
Furthermore, if $p=3$ then the image $\rho_{f,n}$ contains a transvection $\sfm{1}{1}
{0}{1}$.
\end{enumerate}

\item Local conditions.
\begin{enumerate}
\item[(LC1)] If $\fp$ is a prime dividing $p$ then $c(\fp,f)^2\not\equiv \psi(\Frob_\fp)\mod{\wp}$.
\item[(LC2)] Let $\mathfrak{q}$ be a prime dividing the level $\mathfrak{n}$ where
$\rb_f$ is unramified. If $\textrm{Nm}(\mathfrak{q})\equiv 1\mod{p}$
then $p$ divides the order of $\rb_f(Frob_\mathfrak{q})$.

\end{enumerate}
\end{itemize}
Let $k'\geq 2$ be the smallest integer such that $k+k'\equiv 2\mod{(p-1)p^{n-1}}$.
Then $f$ splits  mod ${p^n}$ if and only if it has a $p$-ordinary mod $p^n$ weak
companion form $g$
of weight $k'$ and character $\psi$.
\end{thmm}

The proof, given in section \ref{modular lift}, relies on being able to adapt Taylor's generalisation, \cite{Tay1}, of Ramakrishna's methods, \cite{Ram}, to lift
$\rho_{f,n}\otimes\chi^{1-k}$ to characteristic $0$ with prescribed local properties. Modularity of this lift is established by using results of Skinner and Wiles in
\cite{SW1} along with the existence of companion forms mod $p$ over totally real
fields due to Gee ( \cite[Theorem 2.1]{Gee}). The construction of characteristic $0$
lifts for certain classes of mod ${p^n}$ representations is carried out in section
\ref{lifting to W}. (See Theorem \ref{W lift 2} for the statement.)

The Main Theorem, in practice, is not useful for checking when a given newform
fails to split mod ${p^n}$ because we have very little control over the level of
the weak companion form.  However, as we show in Section 4, in the situation when the dimension of the tangent space $\mathbf{t}_\mathcal{D}$ (associated to a deformation condition $\mathcal{D}$ of $\rb$) is $0$, we can use higher companion forms to computationally verify a conjecture of Greenberg connecting local splitting with complex multiplication. We conclude by giving examples in support of this conjecture.

\section{Toolkit}

The method we use for obtaining a fine structure on deformations of a mod ${p^n}$
representation is an adaptation of the more familiar mod $p$ case and has two key
components:  the existence of sufficiently well behaved local deformations; and, for
the existence of characteristic $0$ liftings, being able to place local constraints so
that the dual Selmer group vanishes. Naturally, both of these present difficulties in
the general mod ${p^n}$ case. In this section, we discuss the tools that will enable us
to manage the difficulties for certain classes of mod ${p^n}$ representations.

Throughout this section $p$ is an odd prime, $\pmb{k}$ is a finite field of
characteristic $p$ and
$W$ is the Witt ring of $\pmb{k}$.

\subsection{Deformations and substantial deformation conditions}
 In the main, we  follow Mazur's treatment of deformations and deformation conditions
in \cite{mazur2}.
 Given a residual representation, a deformation condition is simply a collection of
liftings satisfying some additional
 properties (closure under projections, a Mayer-Vietoris property etc).  The
fundamental consequence then is the
 existence of a (uni)versal deformation.

We expand on this: Suppose we are given  a `nice' profinite group $\Gamma$, and a
continuous representation
$\overline\rho:\Gamma\longrightarrow GL_2(\pmb{k})$. If $\mathcal{D}$ is a deformation
condition for
$\overline\rho$ then there is a complete local Noetherian $W$-algebra $R$ with residue
field $\pmb{k}$ and a lifting
$\rho:\Gamma\longrightarrow GL_2(R)$ in $\mathcal{D}$ with the following property:
If $\rho':\Gamma\longrightarrow GL_2(A)$ is a lifting of $\overline\rho$ in
$\mathcal{D}$ then there is a morphism
$R\longrightarrow A$ which gives, on composition with $\rho$, a representation strictly
equivalent to $\rho'$.
In addition, we require that the morphism above is unique when $A$ is the ring of dual
numbers
$\pmb{k}[\epsilon]/(\epsilon^2)$. If the projective image of $\overline\rho$ has
trivial centralizer then $R$,
together with $\rho$, represents the functor that assigns type $\mathcal{D}$
deformations to a coefficient ring.
We shall use the natural identification of  the tangent space
$\mathbf{t}_\mathcal{D}$ with a subspace of
$H^1(\Gamma, \Ad\overline\rho)$ (and as a subspace of
$H^1(\Gamma, \Ad^0 \overline\rho)$ when considering deformations
with a fixed determinant).
The (uni)versal
deformation ring $R$ then has a presentation $W[[T_1,\ldots ,T_n]]/J$ where
$n=\D_{\pmb{k}}\mathbf{t}_\mathcal{D}$. We will be particularly interested in smooth
deformation conditions
(so the ideal of relations $J$ will be (0)).

As hinted in the beginning of this section, the method we use for constructing smooth global deformation conditions depends upon
being able to find
  local (uni)versal deformation rings smooth in a large number of variables.  It will
be convenient to make the
  following definition:
\begin{df} Let $F$ be a local field and let
   $\overline\rho:G_F\longrightarrow GL_2(\pmb{k})$ be a residual representation.

 We  call a deformation condition  for $\overline\rho$ with fixed determinant a
\textit{substantial deformation condition}
 if it is smooth and its tangent space $\mathbf{t}$ satisfies the inequality
 \[
 \dim_{\pmb{k}}\mathbf{t}\geq\dim_{\pmb{k}}H^0(G_F,\Ad^0\overline\rho)+[F:\Q_p]\delta
 \]
 where $\delta$ is 1 when $F$ has residue characteristic $p$ and 0 otherwise.

Accordingly, any lift   $\rho: G_F\longrightarrow GL_2(A)$ of $\overline\rho$  in this collection of liftings will be called  a \textit{substantial deformation.}

\end{df}

We now give examples of substantial deformation conditions. (The reader may compare these with examples E1 -- E4 in \cite{Tay1}.) From here on, for the rest of the section, $F$ is a  finite extension of $\Q_l$ for some prime $l$. As in the definition above,
let  $\overline\rho:G_F\longrightarrow GL_2(\pmb{k})$ be a residual representation.

 \begin{example}\label{E1} Assume that the residue characteristic of $F$ is different from $p
$. Suppose that the order of $\rb(I_F)$ is co-prime to $p$, and let $d:G_F
\longrightarrow W^\times$ be
 a  character lifting $\textrm{det}\overline\rho$. The collection of liftings of $\rb$
which factor
 through $G_F/(I_F \cap \mathrm{ker}\rb)$ and have determinant $d$ is a substantial
deformation condition.
 The tangent space has dimension $\D_{\pmb{k}}H^0(G_F,\Ad^0\rb)$.
\end{example}

\begin{example}\label{E2} Suppose that
\[
\rb\sim \fm{\overline\chi  }{ \ast }{ 0 }{1}\overline\varepsilon
\]
for some character
$\overline\varepsilon:G_F\longrightarrow \pmb{k}^\times$. Moreover, assume that if $
\overline\rho$ is semi-simple then
$\overline\chi$ is non-trivial. Fix a character $\varepsilon:G_F\longrightarrow W^
\times$ lifting $\overline\varepsilon$.
Then the collection of liftings strictly equivalent to
\[
\fm{\chi}{\ast }{0}{1}\varepsilon
\]
is a substantial deformation
condition.
Note that $\rb$ is equivalent to a representation of the form  considered above only if
$p$ divides the order of
 $\rb(I_F)$. (See Example 3.3 of \cite{JM}.)
\end{example}

\begin{example}\label{wt k liftings}
We now assume that the residue characteristic of $F$ is $p$.
Suppose we are given an integer $k\geq 2$ and a representation $\rb :G_F\longrightarrow
GL_2(\pmb{k})$  such that
\[
\rb=\fm{\overline\chi^{k-1}\overline\psi_1 }{*}{0}{\overline\psi_2}
\]
where $\overline\psi_1,\overline\psi_2$ are  unramified characters. Let  $\psi$ be the Teichm\"{u}ller
lift of $\overline\psi_1\overline\psi_2$.
If $A$ is a coefficient ring, we shall call a lifting $\rho_A: G_F\longrightarrow GL_2
(A)$ of
$\rb$ a {\em $\overline\psi_2$-good weight $k$ lifting with character $\psi$} if $
\rho_A$ is strictly equivalent to a representation of the form
\[
\fm{\widetilde{\psi_1}\chi^{k-1}} {\ast} { 0} {\widetilde{\psi_2}}
\]
 for some unramified characters
$\widetilde{\psi_1}, \widetilde{\psi_2}:G_F\longrightarrow A^\times$ lifting
$\overline\psi_1, \overline\psi_2$ and $\widetilde{\psi_1} \widetilde{\psi_2}=\psi$.

\end{example}

We then have the following property of weight $k$ liftings (proof immediate, Example
3.4 in \cite{JM}):

\begin{prop}\label{weight k}
Let $\overline\rho: G_F\longrightarrow GL_2(\pmb{k})$ be as above in Example \ref{wt k
liftings}, and further assume that
$\overline\chi^{k-1}\psi_1\neq \overline\chi\psi_2$. Then the deformation condition
consisting of weight $k$ liftings of $\rb$
is a smooth deformation condition. The dimension of its tangent space is equal to
$[F:\Q_p]+\D_{\pmb{k}}H^0(G_F,\Ad^0\rb)$.
\end{prop}

\subsection{Subgroups of $GL_2(W/p^n)$} \label{subgroups}

We now derive some properties of certain subgroups of $GL_2(W/p^n)$ which will be of
relevance in constructing global deformations. Let's recall that $p$ is an odd prime,
and that $\pmb{k}$ is the residue field $W/p$. We denote by $\Ad^0$ the the vector
space of trace $0$ $2\times 2$-matrices over
 $\pmb{k}$ with $GL_2(W/p^n)$ acting by conjugation,
and by $\Ad^0(i)$ its twist by the $i$-th power of the determinant.
For convenience,
we record the following useful identity
\begin{equation}
\label{conj}
\fm{1}{x}{0}{1}\fm{a}{b}{c}{-a}\fm{1}{-x}{0}{1}=
\fm{a+cx}{b-2ax-cx^2}{c}{-a-cx}.
\end{equation} 

\begin{lem} \label{subgroup 1}For $p \geq 3$,

(a)  $H^1(GL_2(\pmb{k}),\Ad^0(i))=0$ if $i=0,1$ and $\pmb{k} = \F_5$

(b)  $H^1(SL_2(\pmb{k}),\Ad^0(i))=0$ if $i=0,1$ and $\pmb{k} \neq \F_5$.

\end{lem}

\begin{proof}
The claim is well known when $\pmb{k}\neq\F_5$ and $i=0$---see Lemma 2.48 of \cite{DDT}, for instance.  The proof, in general, is a simple exercise following the proof of Lemma 1.2 of \cite{Flach}. \end{proof}

\begin{prop} \label{subgroup 2} Let $G$ be a subgroup of $GL_2(W/p^n)$. Suppose the mod
$p$ reduction of $G$ contains $SL_2(\pmb{k})$.
Furthermore, assume that if $p=3$ then $G$ contains a transvection $\sfm {1}{1}{0}{1}$.
Then the following statements hold.
\begin{enumerate}
\item[(a)] $G$ contains $SL_2(W/p^n)$.
    \item[(b)] Suppose that $p\geq 5$. If $\pmb{k}=\F_5$ assume further that $G\mod
{5}=GL_2(\F_5)$.
    Then $H^1(G,\Ad^0(i))=0$ for $i=0,1$.
        \item[(c)] The restriction map
        $H^1\left(G,\Ad^0(i)\right)\longrightarrow H^1\left(\sfm{1}{1}{0}{1} , \Ad^0(i)
\right)$ is an injection
        (for all $p\geq 3$).
            \end{enumerate}
\end{prop}

\begin{proof} 

 \noindent {\it Part (a).} We shall only verify that if $G$ is a subgroup of $SL_2(W/p^n)$ whose mod $p^{n-1}$
reduction is
$SL_2(W/p^{n-1})$ then $G=SL_2(W/p^n)$. The kernel of the reduction map
$G\longrightarrow SL_2(W/p^{n-1})$ consists of matrices of the form $I+p^{n-1}A$ with
$A$ an element of some additive
 subgroup of $\Ad^0$ stable under the action of $G$. Consequently either $G=SL_2(W/p^n)
$ or else the reduction map
  $G\longrightarrow SL_2(W/p^{n-1})$ is an isomorphism. We will now discount the second possibility. So suppose that $G\longrightarrow SL_2(W/p^{n-1})$ is an isomorphism. For $p \geq 5$ it follows from  \cite{Se}, IV-23 Lemma 3, that $G \subset  GL_2(W/p^n)$ contains a transvection  $\sfm{1}{1}{0}{1}$ (simply consider $G \cap GL_2(\Z/p^n)$). For $p = 3$ this is true by hypothesis. However, $\sfm{1}{1}{0}{1}$ has order $p^n$ in G and order $p^{n-1}$ in $SL_2(W/p^{n-1})$---a contradiction.

\medskip
\noindent {\it Part (b).}  The hypothesis and Lemma \ref{subgroup 1} implies that $H^1\left( G\mod{p}, \Ad^0(i)
\right)=0$. So let's assume that
$n\geq 2$ and that $H^1\left(G\mod{p^{n-1}},\Ad^0(i)\right)=0, $ and suppose  that
$0\neq \xi\in H^1(G,\Ad^0(i))$.
Then the restriction $\xi$ to
 $H:=\ \text{ker}\left(G\longrightarrow G\mod{p^{n-1}}\right)$ is a
 group homomorphism compatible with the action of $SL_2(W/p^{n-1})$.
 It follows from part (a) that $H$ is in fact
 $\text{ker}\left(SL_2(W/p^n)\longrightarrow SL_2(W/p^{n-1})\right)$.
  Since  $H$ is naturally identified with $\Ad^0$,  it follows that $\xi\vert_H$ is an
  isomorphism. Let's denote by $W^\nu$ the ring $W/p^{n-1}\oplus \pmb{k}\e$ where
  $\e^2=p\e=0$. (Or equivalently $W^\nu\cong W[\e]/(p^{n-1}, \e^2, p\e)$.) We then see
  that the homomorphism
  $SL_2(W/p^n)\longrightarrow SL_2(W^\nu)$ given by
\[
g\longrightarrow \left(I+\e\xi(g)\right)(g\mod{p^{n-1}})
\]
is an isomorphism.
To finish off, we proceed as in part (a): The transvection
 $\sfm {1}{1}{0}{1}\in  SL_2(W^\nu)$ has order $p^{n-1}$
  while its pre-image
 in $SL_2(W/p^n)$, a matrix of the form
 $(I+p^{n-1}\sfm {a}{b}{c}{-a})\sfm {1}{1}{0}{1}$,  has order $p^n$.

\medskip \noindent {\it Part (c).} Suppose $0\neq \xi\in H^1(G,\Ad^0(i))$ restricts
to a trivial cohomology class in
$H^1\left(\sfm{1}{1}{0}{1}, \Ad^0(i)\right)$. Then the restriction of
 $\xi$ to $\sfm {1}{p^{n-1}}{0}{1}$ is trivial.
 Set
 $N:=\ \text{ker}\left(G\longrightarrow G\mod{p^{n-1}}\right)$. Then $\xi\vert_N$ has a
non-trivial kernel, and hence $\xi|_N$ is trivial.
 Thus $\xi$ is  a non-zero element of
$H^1(G\mod{p^{n-1}},\Ad^0(i))$. We are thus reduced to the case when $n=1$. Now
$H^1(SL_2(\pmb{k}), \Ad^0)=0$ except
 when $\pmb{k}=\F_5$,  so we are reduced to the case when $G$ is a subgroup of
 $GL_2(\F_5)$ containing $SL_2(\F_5)$.
 But in this case $\left(\sfm{1}{1}{0}{1}\right)$ is the Sylow 5-subgroup of $G$, and
hence if
  $\xi|_{\sfm{1}{1}{0}{1}}=0$ then $\xi =0$.
\end{proof}

\section{Constructing characteristic $0$ lifts of mod ${p^n}$ Galois representations}
Firstly, we elaborate on the remark made in the introduction that if $\rho_{f,\wp}$ mod $\wp$ is absolutely irreducible then
the splitting (or not) of $\rho_{f,\wp}\mod{\wp^n}$ is independent of the  choice of a
lattice used to define $\rho_{f,\wp}$.
Indeed, if for some $M\in GL_2(K_{f,\wp})$ the
conjugate $M\rho_{f,\wp}M^{-1}$ is integral
and stabilises the upper triangular decomposition group $G_\mathfrak{p}$, then $M$ is a
scalar multiple of $\sfm{u}{v}{0}{1}$ where $u\equiv 1\mod{\wp}$ and $v\equiv 0\mod
{\wp}$. If we denote by
$c_\mathfrak{p}\in H^1(G_\mathfrak{p}, \mathcal{O}_{f,\wp}(\psi_{1\mathfrak{p}}\psi_
{2\mathfrak
{p}}^{-1}\chi^{k-1}))$ the cohomology class for $\rho_{f,\wp}|_{G_\mathfrak{p}}$, then
the
cohomology class for the extension at $\mathfrak{p}$ determined by
$M\rho_{f,\wp}M^{-1}$ is $uc_\mathfrak{p}$. Hence, if $\rho_{f,\wp}$ mod $\wp$ is
absolutely irreducible we can speak of $\rho_{f,\wp}\mod{ \wp^n}$ being split without
any ambiguity.

We can now formulate precise conditions under which a given mod ${p^n}$ Galois representation can be lifted to characteristic $0$, and use the lifts  constructed to prove the existence of weak companion forms.

\subsection{Deformations of mod ${p^n}$ representations to $W(k)$}
\label{lifting to W}

We now suppose we are given a totally real number field $F$  and  continuous
odd representations $\rb:G_F\longrightarrow GL_2(\pmb{k}),\ \rho_n:G_F\longrightarrow
GL_2(W/p^n), n\geq 2$,
with $\rb=\rho_n\mod{p}$.    We shall also assume that the $\rb, \rho_n$ satisfy the
following.
\begin{hypothesis a}\label{hypothesis a}
The image of $\rb$ contains $SL_2(\pmb{k})$. Furthermore,  if $p=3$ then the image
 of $\rho_n$ contains the transvection $\sfm{1}{1}{0}{1}$.\end{hypothesis a}

Fix a character $\epsilon:G_F\longrightarrow W^\times$ lifting the determinant of
 $\rho_n$.
We wish to consider global deformation conditions $\mathcal{D}$ for $\rb$ with
 determinant $\epsilon$ such that $\rho_n$ is a deformation of type $\mathcal{D}$. We
shall
abbreviate this and call $\mathcal{D}$ a deformation condition for $\rho_n$. Except for
a change in choice of lettering for primes of $F$  we keep
the notation of \cite{JM}. Thus $\mathcal{D}_\fq$ is the local component at a prime
$q$,
$\mathbf{t}_{\mathcal{D}_\fq}$ is the tangent space there, and
$\mathbf{t}_{\mathcal{D}_\fq}^\perp\subseteq H^1(G_{F_{\fq}},\Ad^0\rb(1))$ is the orthogonal complement of $\mathbf{t}_{\mathcal{D}_\fq}$ under the pairing induced by
 \[
 \Ad^0\rb\times \Ad^0\rb(1)\xrightarrow{\text{trace}}\pmb{k}(1).
 \]
 The tangent space for
$\mathcal{D}$ is the Selmer group
$H^1_{\{\mathbf{t}_{\mathcal{D}_\fq}\}}(F,\Ad^0\rb)$; the dual Selmer group
$H^1_{\{\mathbf{t}_{\mathcal{D}_\fq}^\perp\}}(F,\Ad^0\rb(1))$ is determined by the local conditions $\mathbf{t}_{\mathcal{D}_\fq}^\perp$. (See for instance \cite[Definition 8.6.19]{NSW}.) We also set
\[
\delta(\mathcal{D}):=
\text{dim}_{\pmb{k}}H^1_{\{\mathbf{t}_{\mathcal{D}_\fq}\}}(F,\Ad^0\rb)-
\text{dim}_{\pmb{k}}H^1_{\{\mathbf{t}_{\mathcal{D}_\fq}^\perp\}}(F,\Ad^0\rb(1)).
\]

\begin{prop}\label{W lift 1}
Suppose we are given a deformation condition $\mathcal{D}$  for $\rho_n$
with determinant $\epsilon$. Let $S$ be a fixed finite set of primes of $F$
including primes where $\mathcal{D}$ is ramified and all the infinite primes.
If $\delta(\mathcal{D})\geq 0$ we can find a deformation condition $\mathcal{E}$ for $
\rho_n$
with determinant $\epsilon$ such that:
\begin{itemize}
\item The local conditions $\mathcal{E}_\fq$ and $\mathcal{D}_\fq$ are the same at
primes
$\fq\in S;$
\item $\mathcal{E}_\fq$ is a substantial deformation condition for $\fq\notin S;$ and,
\item $H^1_{\{\mathbf{t}_{\mathcal{E}_\fq}^\perp\}}(F,\Ad^0\rb(1))=(0)$.

\end{itemize}
\end{prop}

\begin{proof} Let $K$ be the splitting field of $\rho_n$ adjoined $p^n$-th roots
of unity. We claim that we can find elements
$g,h\in \text{Gal}(K/F)$ such that
\begin{enumerate}
\item[(R1)] $\rho_n(g)\sim\sfm{-1}{0}{0}{1}$ and $\chi(g)=-1\mod{p^n};$
\item[(R2)] $\rho_n(h)\sim a\sfm{1}{1}{0}{1}$ and $\chi(h)=1\mod{p^n}$.
\end{enumerate}
For R1, we can take $g$ to be complex conjugation. For R2, by considering
$\epsilon=\chi (\epsilon\chi^{-1})$ or otherwise, we can write
$\epsilon=\chi\epsilon_0\epsilon_1^2$ where $\epsilon_0$ is a finite order character of
order
co-prime to $p$. Our assumptions on the size of $\rb$ and $\rho_n$ (when $p=3$)   along
with Proposition \ref{subgroup 2}   imply  that the image of  the twist of $\rho_n
\otimes \epsilon_1^{-1}$
contains $SL_2(W/p^n)$. Thus we can find $h_1\in \text{Gal}(K/F)$ such that
$\rho_n(h_1)=\sfm{1}{1}{0}{1}\epsilon_1(h_1)$ and we get $\epsilon_0(h_1)\chi(h_1)=1$.
We can then take $h$ to be $h_1^{p^k-1}$ where $p^k$ is the cardinality of $\pmb{k}$.

 We first adjust  $\mathcal{D}$ and define a deformation
 condition $\mathcal{E}_0$ for $\rho_n$ with determinant $\epsilon$ as follows.
We make no change if $p\geq 5$ and the projective image of $\rb$ strictly contains
$PSL_2(\F_5);$
so $\mathcal{E}_0$ is $\mathcal{D}$. Now for the remaining cases: Suppose that either
$p=3$ or  the projective image
 of $\rb$ is $A_5$ (so $\pmb{k}$ is necessarily $\F_5$).
Using the Chebotarev Density Theorem and R2 above, we can find a prime $\fq_0\notin S$
with $\fq_0\equiv 1\mod{p^n}$  and
$\rho_n(\Frob_{\fq_0})=a\sfm{1}{1}{0}{1}$. Let
 $\mathcal{E}_0$ be the deformation condition of $\rb$ with determinant $\epsilon$
 characterized by the following local conditions:
\begin{itemize}
\item at primes $\fq\neq \fq_0$, $\mathcal{E}_{0\fq}=\mathcal{D}_\fq;$
\item at $\fq_0$, $\mathcal{E}_{0\fq_0}$ consists of deformations of the form
\[
\fm{\chi}{*}{0}{1}\epsilon'
\]
 where $\epsilon':G_{v_0}\longrightarrow W^\times$ is
unramified and $\epsilon|_{G_{\fq_0}}=\chi{\epsilon'}^2$.
\end{itemize}
By our choice of $\fq_0$, $\mathcal{E}_0$ is a deformation condition
for $\rho_n$. Further,  $\mathcal{E}_{0\fq_0}$ is a substantial deformation and all
non-zero cohomology classes in
$\mathbf{t}_{\mathcal{E}_{0\fq_0}},   \mathbf{t}_{\mathcal{E}_{0\fq_0}}^\perp$ are
ramified.

We claim that the restriction maps
\[
H^1_{\{ \mathbf{t}_{\mathcal{E}_{0\fq}}\}}(F,\Ad^0\rb)\longrightarrow H^1(G_K,\Ad^0\rb)
\quad
\text{and}\quad
 H^1_{\{\mathbf{t}_{\mathcal{E}_{0\fq}}^\perp\}}(F,\Ad^0\rb(1))\longrightarrow H^1(G_K,
\Ad^0\rb(1))\]
  are injective.
When $p\geq 5$ and the projective image of $\rb$ strictly contains $A_5$ an easy
calculation using Proposition \ref{subgroup 2} shows that
$H^1(\text{Gal}(K/F),\Ad^0\rb)$ and $H^1(\text{Gal}(K/F),\Ad^0\rb(1))$ are trivial, and
so the injectivity follows.
In the case when $p=3$ or  the projective image
 of $\rb$ is $A_5$, we argue as follows: If
 $\xi\in\ \text{ker}\left(
 H^1_{\{ \mathbf{t}_{\mathcal{E}_{0\fq}}\}}(F,\Ad^0\rb)\longrightarrow H^1(G_K,
\Ad^0\rb)\right)$, then $\xi$ is naturally
 an
 element of $H^1(\textrm{Gal}(K/F),\Ad^0\rb)$. Thus $\xi$ is unramified at $\fq_0$ and
 so the restriction of $\xi$ to
the decomposition group at $\fq_0$ must be trivial. Using Proposition \ref{subgroup 2}
it follows
that
 $\xi\in  H^1(\textrm{Gal}(K/F),\Ad^0\rb)$ is trivial. A similar argument works for
 $\Ad^0\rb(1)$.

The proof is now standard: If the dual Selmer group for $\mathcal{E}_0$ is non-trivial
then we can find
\[
0\neq \xi\in H^1_{\{ \mathbf{t}_{\mathcal{E}_{0\fq}}\}}(F,\Ad^0\rb), \quad 0\neq \psi
\in
 H^1_{\{\mathbf{t}_{\mathcal{E}_{0\fq}}^\perp\}}(F,\Ad^0\rb(1)).
 \]
  Take $g\in \text{Gal}(K/L)$ as in R1, consider
 pairs $(M_1,N_1), (M_2,N_2)$ where
 $\left\{\sfm{0}{*}{*}{0}\right\}=N_1\subset M_1=\Ad^0\rb$, $\left\{\sfm{*}{*}{0}{*}
\right\} =N_2\subset M_2=\Ad^0\rb(1)$
and apply Proposition 2.2 of \cite{JM}. One can then find a
 prime $\fr\notin S\cup\{\fq_0\}$ lifting $g$ such that the restrictions of $\xi, \psi$
to $G_\fr$ are not in
 $H^1(G_\fr,N_1), H^1(G_\fr,N_2)$.

Now take $\mathcal{E}_1$ to be the deformation condition with determinant $\epsilon$
as follows: $\mathcal{E}_1$ and $\mathcal{E}_0$ differ only at $\fr$, and at $\fr$,
the local component consists of deformations of the form
 $\sfm{\chi}{*}{0}{1}(\epsilon/\chi)^{1/2}$ considered in Example \ref{E2}. Here,
 $(\epsilon/\chi)^{1/2}$ is the unramified character determined by taking the
  square-root of $\epsilon(\Frob_\fr)\chi^{-1}(\Frob_\fr)$.
 Since $\Frob_\fr$ lifts $g$ we have $\chi(\Frob_{\fr})\equiv -1\mod{p^n}$, and
 consequently $\mathcal{E}_1$ is a substantial deformation condition for $\rho_n$.
  The rest is identical to the proof of
  Proposition 4.2, \cite{JM}: The dual Selmer group for $\mathcal{E}_1$ has dimension
one less than that of the dual
   Selmer group for
 $\mathcal{E}_0$. (Of course $\delta(\mathcal{E}_1)= \delta(\mathcal{E}_0)= \delta
(\mathcal{D})$.)
\end{proof}

We can now prove a general result for lifting a mod ${p^n}$ representation to
characteristic $0$.

\begin{thm} \label{W lift 2}
Let $\mathcal{D}$  be a deformation condition for $\rho_n$
with determinant $\epsilon$, and let $S$ be a fixed finite set of primes of $F$
including primes where $\mathcal{D}$ is ramified and all the infinite primes.
Suppose that each local component is substantial.
We can then find a deformation condition $\mathcal{E}$ for $\rho_n$
with determinant $\epsilon$ such that:
\begin{itemize}
\item The local conditions $\mathcal{E}_\fq$ and $\mathcal{D}_\fq$ are the same at
primes
$\fq\in S;$
\item Each local component is a substantial deformation condition;
\item The dual Selmer group $H^1_{\{\mathbf{t}_{\mathcal{E}_\fq}^\perp\}}(F,\Ad^0\rb
(1))$ is trivial.
\end{itemize}
$\mathcal{E}$ is a smooth deformation condition and the universal deformation ring
is a power series ring over $W$ in
$\delta(\mathcal{D})$ variables. In particular, there is
a representation $\rho:G_F\longrightarrow GL_2(W)$ of type $\mathcal{E}$ lifting
$\rho_n$.
\end{thm}

\begin{proof} The only verification required is to check that
$\delta(\mathcal{D})\geq 0$ and that
$\text{dim}_{\pmb{k}}H^1_{\{\mathbf{t}_{\mathcal{E}_\fq}\}}(F,\Ad^0\rb)=
\delta(\mathcal{E})=\delta(\mathcal{D})$. This is done using Wiles' formula
(cf \cite[Theorem 8.6.20]{NSW}).
\end{proof}

\subsection{Modular characteristic $0$ lifts and proof of Main Theorem}
\label{modular lift}

We now look at the question of producing characteristic zero liftings which are
modular.  Given a mod ${p^n}$ Galois representation
$\rho_n:G_F\longrightarrow GL_2(W/p^n)$ with $\rho_n\mod{p}$ modular,
 when can we guarantee the existence
of a modular form $f$ with $\rho_{f,p}\mod{p^n}\sim \rho_n?$ Our answer is a modest
attempt using Theorem \ref{W lift 2} to produce a characteristic $0$ lift and then
invoking results of Skinner and Wiles \cite{SW1} to prove that it is modular.

For the rest of this section,   $F$ is a totally real field and
$\psi: G_F\longrightarrow W^\times$ is a
finite order character of $G_F$ unramified at primes dividing $p$.

\begin{prop}\label{modular lift 1}
Let $\rho_n:G_F\longrightarrow GL_2(W/p^n)$ be a  continuous odd
representation satisfying  \ref{hypothesis a}. Suppose
$\epsilon:=\psi\chi^a, a\geq 1$ lifts the determinant of $\rho_n$. Assume that:
\begin{itemize}
\item[(i)] At a prime $\fq\nmid p$ where  $\rho_n$ is ramified,  the restriction
$\rho_n|_{G_\fq}$ is substantial there
  and that a substantial deformation
condition $\mathcal{D}_\fq$ is specified for $\rho_n$.
\item[(ii)] At a prime $\fp$ dividing $p$,
\[
\rho_n|_{G_\fp}\sim \fm{\chi^a\psi_{1\fp}}{*}{0}{\psi_{2\fp}}
\]
 where $\psi_{1\fp},\psi_{2\fp}$ are
unramified,
$\chi^{a}\psi_{1\fp}\not\equiv\psi_{2\fp}\mod{p}$  and
$\chi^{a}\psi_{1\fp}\not\equiv\chi\psi_{2\fp}\mod{p}$.
\item[(iii)] There is an ordinary, parallel weight at least 2, modular form which is a
$(\psi_{2\fp}\mod{p})$- good lift of $\rho_n\mod{p}$.
\end{itemize}
There is then a modular form $f$ such that its associated $p$-adic representation
$\rho_{f,p}:G_F\longrightarrow GL_2(W)$  lifts $\rho_n$, has determinant $\psi\chi^a$,
is of   of type $\mathcal{D}_\fq$ at primes $\fq\nmid p$ where $\rho_n$ is ramified,
and
 \[
 \rho_{f,p}|_{G_\fp}\sim \fm{\psi_{1\fp}'\chi^a}{*}{0}{\psi_{2\fp}'}
 \]
  at primes $\fp|p$ with $\psi_{2\fp}'$ an unramified lift of $\psi_{2p}\mod{p}$.

\end{prop}

\begin{proof}
At a prime $\fp |p$ take $\mathcal{D}_\fp$ to be the class of deformations  of the
form
\[
\fm{\psi_{1\fp}'\chi^a}{*}{0}{\psi_{2\fp}'}
\]
 where $\psi_{1\fp}'$ (resp. $\psi_{2\fp}'$) is
an unramified lifting of $\psi_{1\fp}\mod{p}$ (resp. $\psi_{2\fp}\mod{p}$), and
$\psi_{1\fp}\psi_{2\fp}=\psi$. This is a substantial deformation for $\rho_n$ at $\fp$
by Proposition \ref{weight k}. By Theorem \ref{W lift 2}, there is  a smooth
deformation
 condition $\mathcal{E}$
for $\rho_n$ which agrees with $\mathcal{D}_\fp$ at primes above $p$ and
primes where $\rho_n$ is ramified. Thus there is continuous representation
$\rho:G_F\longrightarrow GL_2(W)$ with $\rho\mod{p^n}=\rho_n$,  unramified outside
finitely many primes, determinant $\psi\chi^a$ and
$\rho|_{I_\fp}\sim \sfm{\chi^a}{*}{0}{1}$ at primes $\fp |p$. The proposition now
follows
from Skinner-Wiles \cite{SW1}.
\end{proof}

\begin{proof}[Proof of Main Theorem] Let's recall the set up: We are given
a Hilbert modular newform $f$ of weight $k\geq 2$ character $\psi$ which is ordinary
at $p$ and whose reduction mod ${p^n}$ gives
$\rho_{f,n}:G_F\longrightarrow GL_2(W/p^n)$. For each prime   $\mathfrak{p}$  of $F$ over $p,$ let $\psi_{1\mathfrak
{p}},\psi_{2\mathfrak{p}}$ be the  unramified characters such that
\[
\rho_f|_{G_\mathfrak{p}}\sim \fm{\psi_{1\mathfrak{p}}\chi^{k-1}}{*}{0}{\psi_
{2\mathfrak{p}}}.
\]
As $\psi_{1\fp}\psi_{2\fp}=\psi$ and $\psi_{2\fp}(\Frob_{\fp})=c(\fp,f)$, hypothesis
LC1 ensures  $\psi_{1\mathfrak{p}},\psi_{2\mathfrak{p}}$ are distinct modulo ${\wp}$.
From this, one deduces easily that if $f$ has a weak
companion form mod ${p^n}$ then $\rho_{f,n}$ splits at $p$. We now show that `split at
$p$' implies the existence of a weak companion form.

Let $\rho_n:=\rho_{f,n}\otimes\chi^{1-k}$, and set $\rb:=\rho_n\mod{p}$. Recall that
$k'\geq 2$ is the smallest integer satisfying the congruence
$k+k'\equiv 2\mod{(p-1)p^{n-1}}$. Define a global deformation condition $\CD$ for
$\rb\otimes \chi^{1-k}\mod{p}$  by the following requirements:
\begin{enumerate}
\item[(a)] Deformations are unramified outside primes dividing $p\fn$ and have
determinant $\psi\chi^{k'-1}$.
\item[(b)] At a prime $\fp\mid p$, the local condition $\CD_\fp$ consists of
deformations
of the form
\[
\fm{\psi_{2\fp}'\chi^{k'-1}}{*}{0}{\psi_{1\fp}'}
\]
 where $\psi_{1\fp}'$ (resp. $\psi_{2\fp}'$) is
an unramified lifting of $\psi_{1\fp}\mod{p}$ (resp. $\psi_{2\fp}\mod{p}$), and
$\psi_{1\fp}\psi_{2\fp}=\psi$.

\item[(c)] Let $\fq$ be a prime dividing $\fn$, the level of $f$. We need to
distinguish
two cases:
\begin{enumerate}
\item[(i)] If $\fq$ does not divide the conductor of $\psi$ then
$\rb\vert_{G_\fq} \sim\sfm{\overline\chi}{*}{0}{1}\bar\epsilon$
 for some character $\bar\epsilon$.  Further, hypothesis LC2 ensures that if
 $\rb\vert_{G_\fq}$ is semisimple then $\overline\chi\neq 1$. We then take $\CD_\fq$ to
be local liftings with determinant $\psi\chi^{k'-1}$ of the type considered in Example
\ref{E2}.
\item[(ii)] If $\fq$ divides the conductor of $\psi$ then  $\rho_f(I_\fq),
\rb(I_\fq)$ are finite and have the same order. In this case we take $\CD_{\fq}$  as in
Example \ref{E1} i.e.
lifts with detereminant $\psi\chi^{k'-1}$ which factor through
$G_\fq/(I_\fq\cap\ker\rb)$.

\end{enumerate}
\end{enumerate}

It then follows that $\rho_n$ is a deformation of type $\CD$ and that at each prime
$\fq\nmid p$ where $\rho_n$ is ramified the local deformation condition $\CD_\fq$ is
substantial there. As $p$ is unramified in $F$, the distinctness of
$\psi_{1\fp},\psi_{2\fp}$ modulo $p$ implies that $\rb$ satisfies hypothesis (ii)
of Proposition \ref{modular lift 1}. From the existence of mod ${p}$ companion forms,
  (\cite[Theorem 2.1]{Gee}), it follows that $\rb$ has an ordinary modular lift which is
$(\psi_{1\fp}\mod{p})_{\fp\vert p}$-good. The existence of a mod ${p^n}$ weak companion
form $g$ for $f$ of weight $k'$ character $\psi$ now follows from
Proposition \ref{modular lift 1}.
\end{proof}

\section{Checking local splitting: A computational approach} \label{splitting}

The lifting result of the previous section is not suitable for computational purposes
in general because, except in the case when dual Selmer group was already trivial,
we had no control of the level. There is, however, one case when we do have absolute control. We now describe
 this situation and go on to verify examples of local splitting.

\subsection{A special case}\label{Hecke}

Suppose $\overline\rho: G_\mathbb{Q}\longrightarrow GL_2(\pmb{k})$ is absolutely irreducible and $\CD$ is a deformation condition for $\overline\rho$ such that its tangent space is $0$ dimensional. Then the universal deformation ring $R_{\CD}$  is a quotient of $W(\pmb{k})$. If we also knew that there is a characteristic $0$ lift of type $\CD$, then we must have $R_{\CD} \simeq W$. Consequently any mod ${p^n}$ representation of type $\CD$ lifts to characteristic $0$.

The question now is: How can one check if the tangent space is $0$ dimensional? Observe that we must necessarily have exactly one characteristic $0$ lift of type $\CD$. This alone might not be enough though. For instance, $R_\CD$ might be $W[X]/(X^2)$.

To proceed further, and with the examples we have in mind, we shall assume that $\overline\rho: G_\mathbb{Q}\longrightarrow GL_2(\pmb{k})$ is an absolutely irreducible representation with determinant $\overline\chi$ such that

\begin{itemize}
\item $\rb\vert_{G_p}\sim\sfm{\overline\chi\psi^{-1}}{*}{0}{\psi}$, with $\psi$ unramified and $\psi \neq \psi^{-1}$,
\item if $q \nmid p$ then $\#\rb(I_q)|p$.
\end{itemize}
By Lemma 3.24 of \cite{DDT}, the restriction $\overline\rho|_{G_L}$ is absolutely irreducible where $L=\mathbb{Q}(\sqrt{(-1)^{(p-1)/2}}p)$.

Let $N$ be the Artin conductor of $\rb$. For an integer $k \geq 2$, let $S(k,N,\rb)$ be the (possibly empty) set of newforms of level $N$ with $\rho_f \mod{p} \simeq \rb$.

With notation as in Theorem 3.42 of \cite{DDT}, we then have an isomorphism
$R_\emptyset\xrightarrow{\sim}\mathbb{T}_\emptyset$, where $R_\emptyset$ is the
universal deformation ring for minimally ramified ordinary lifts and
$\mathbb{T}_\emptyset$ is the reduced Hecke algebra generated by the Fourier
coefficients of newforms in $S(2,N,\rb)$. In particular, the dimension of the tangent
space in the minimally ramified case is $0$ if and only if $\#S(2,N,\rb) =1$.

For $n\geq 1$ set $k_n:= (p-1)p^{n-1}-(p-1)+2$ and define a deformation condition
 $\CD_{k_n}$ for $\rb$ as follows: A lift $\rho: G_\mathbb{Q}\longrightarrow GL_2(A)$
 is a deformation of type $\CD_{k_n}$ if
\begin{itemize}
 \item  det $\rho = \chi^{k_n-1}$ and $\rho$ is unramified outside primes dividing $N$,
\item at primes $q|N$, $\rho\vert_{G_q}\sim\sfm{\chi}{*}{0}{1}$ up to twist, and
\item at $p$, $\rho\vert_{G_p}\sim\sfm{\tilde\psi^{-1}\chi^{k_n-1}}{*}{0}{\tilde\psi}$, where $\tilde\psi$ is an unramified lift of $\psi$.
\end{itemize}

Note that for $n=1$ the universal deformation ring $R_{\CD_{k_n}}$ is
 $R_\emptyset \simeq \mathbb{T}_\emptyset$. Clearly, the type $\CD_{k_n}$ deformations
 to $\pmb{k}[\epsilon]/(\epsilon^2)$  are in bijection with type $\CD_2$ deformations.
 Hence if the tangent space of $\CD_2$ has dimension $0$ then so does $\CD_{k_n}$. We
 conclude that $R_{\CD_{k_n}} \simeq W$, corresponding to a unique newform in $S(k_n,N,\rb)$.

\begin{prop}\label{split} Let $f$ be a newform of weight $k \geq 2$, level $N$, trivial
character and ordinary at $p$, such that

\begin{itemize}

\item$\rb_f$ is absolutely irreducible,

\item the conductor of $\rb_f$ is $N$,

\item $\rb\vert_{G_p}\sim\sfm{*}{*}{0}{\psi}$ with $\psi$ unramified and $\psi^2 \neq 1$,

\item if $q \nmid p$ then $\#\rb_f(I_q)|p$.
\end{itemize}
Assume that $p-1$ divides $k$ and that $f$ has exactly one companion form mod ${p}$ of level $N$. Then
$\rho_f$ splits mod ${p^n}$ iff $f$ has a companion form mod ${p^n}$ of level $N$.

\end{prop}

\begin{proof} We need to explain why splitting implies the existence of a companion form. Set $\rb:=\rb_f\otimes \bar\chi^{1-k}$.  By hypothesis $\rb$ is modular and, in fact, $\#S(2,N,\rb) =1$. Therefore, as noted above, the tangent space of $\CD_2$ has dimension $0$. The preceding discussion thus shows that $R_{\CD_{k_n}} \simeq W$, corresponding to a unique newform $g_n$ in $S(k_n,N,\rb)$. Now, if $\rho_f$ splits mod $p^n$ then $\rho_f\otimes\chi^{1-k}$ is a deformation of type $\CD_{k_n}$ and hence $\rho_{g_n}\sim \rho_f\otimes\chi^{1-k}\mod{p^n}$.\end{proof}

\subsection{Greenberg's conjecture}

In this section we give some examples of the existence (or non-existence) of higher
companion forms. We shall restrict ourselves to the setting of classical elliptic
modular forms as we only give examples in this case.

Recall that  a   newform $f$   is said
 to have \textbf{complex multiplication}, or just CM, by a
quadratic character $\phi:G_\Q\longrightarrow \{\pm 1\}$ if
$T(q)f=\phi(\text{Frob}_{q})c(q,f)f$  for almost all primes
$q$.
We  will also refer to CM by the
corresponding quadratic extension. It is well known that a modular form has CM if and
only if its
associated $p$-adic representation is induced from an algebraic Hecke character.

Let $\CO$ be a $p$-adic integer ring  with uniformizer $\pi$ and residue field $\pmb{k}$. Suppose the newform
$f$ is $p$-ordinary and $\rho_f:G_\Q\longrightarrow GL_2(\CO)$. Then, as indicated in
the introduction, $\rho_f|G_p$ can be assumed to be upper triangular with an unramified
lower diagonal entry and this leads to the natural question of determining when
 $\rho_f$ splits at $p$.

 There is a  well-known conjectural connection between  $\rho_f$ to be split at $p$
and $f$ to have CM. The antecedents are sketchy, but Hida  \cite{Hi}   calls it
Greenberg's local non-semisimplicity conjecture; we will simply refer to it as
Greenberg's conjecture which asserts:
   \begin{quote}
   If $f$ is ordinary at $p$ and $\rho_f$ splits at $p$ then $f$ has complex
multiplication.
   \end{quote}
 This is   satisfactorily  known for modular forms over $\Q$
of weight $2$.
  For higher weights, the question remains largely unresolved although some interesting
results involving Hida families are
   shown in Ghate \cite{G}  which also has a survey of results for weight $2$. The
analogous problem for  $\Lambda$-adic  modular forms was resolved in Ghate-Vatsal
\cite{GV} by using deformation theory but similar methods appear not to bear fruit
in the classical case. Emerton \cite{Em} shows how this conjecture would follow from a
$p$-adic version of the variational Hodge conjecture. Through the main theorem and proposition \ref{split},  higher
congruence companions offer a slightly different perspective to the question of $\rho_f
$ splitting at $p$.

To describe this further, let $N$ be the level of the $p$-ordinary newform $f$. Assume
that $f$ has trivial character and has weight $p-1$. For each
positive integer $n$ we set $k_n:=p^{n-1}(p-1)-(p-1)+2$. We then
proceed as follows:
\begin{enumerate}
\item[(a)] Check that $f$ has a companion form mod ${p}$. Check congruences to make sure
that the residual representation $\rho_f\mod{\pi}$ is absolutely irreducible and that
$c(p,f)\not\equiv \pm 1 \mod{\pi}$. We can therefore write
\[
\rho_f\mod{\pi}=\fm{\overline{\chi}^{k-1}\overline{\psi}}{0}{0}{\overline{\psi}^{-1}}
\]
with $\overline{\psi}^{-1}\neq\overline{\psi}$.
    \item[(b)] In order to be able to  check fewer cases, ensure that
    \ $\rho_f$ is minimally ramified i.e. the Artin conductor of $\rho_f\mod
{\pi}$ is $N$.

\item[(c)] Set $\rb:=\rho_f\otimes\chi^{1-k}\mod{\pi}$. For each $n\geq 1$ let $\CD_{k_n}$
be the weight $k_n$, trivial character deformation condition as described in section 4.1.
 Check if the tangent spaces can be taken to be $0$ dimensional. Thus we have to check if $f$ has precisely one companion form of type $\CD_{k_1}$. We then apply proposition \ref{split} to deduce that
         $\rho_f$ splits mod ${p^n}$ if and only if $f$ has a companion form
          mod ${p^n}$
\textit{i.e.} there is a newform $g$ of level $N$, trivial character, weight $k_n$ such
that $f\equiv g\otimes \chi^{k-1}\mod{p^n}$
\end{enumerate}

We check Greenberg's conjecture explicitly for two known non-$CM$ forms of weight $4$. The computations were done on MAGMA.  In both cases $p=5$. We note that in these examples, taking $N$ to be the level of $f$, one may check its companionship with a form $g$ ``\textit{by hand}" by simply verifying the congruences $c(f,m) \equiv m^3c(g,m) \mod5^n$ for $(m,5N)=1$  up to the Sturm bound.

\begin{example}\label{level 21}  Let $f$ be the newform of weight $4$, level $21$ and trivial character with the following Fourier
expansion:

\begin{center}
$g = q - 3q^2 - 3q^3 + q^4 - 18q^5 + 9q^6 + 7q^7 + 21q^8 + 9q^9 + 54q^{10} -
36q^{11} - 3q^{12} - 34q^{13} - 21q^{14} + 54q^{15} - 71q^{16} + 42q^{17} - 27q^{18} -
124q^{19} - 18q^{20} - 21q^{21} + 108q^{22} - 63q^{24} + 199q^{25} +  102q^{26} - 27q^{27} + 7q^{28} + 102q^{29} - 162q^{30 }- 160q^{31} + 45q^{32} + 108q^{33} - 126q^{34} - 126q^{35} + 9q^{36} + 398q^{37} + 372q^{38} + 102q^{39} - 378q^{40} - 318q^{41} + 63q^{42} - 268q^{43} -
    36q^{44} - 162q^{45} + 240q^{47} + 213q^{48} + 49q^{49} - 597q^{50} +\cdots $
    \end{center}

    MAGMA outputs modulo $5$,  a unique companion form $g$  weight $2$, level $21$ and trivial character with the following Fourier expansion:

  \begin{center}
$f = q - q^2 + q^3 - q^4 - 2q^5 - q^6 - q^7 + 3q^8 + q^9 + 2q^{10} + 4q^{11} - q^{12} -
2q^{13} + q^{14} - 2q^{15} - q^{16}- 6q^{17}- q^{18} + 4q^{19} + 2q^{20} - q^{21} - 4q^{22} + 3q^{24} - q^{25} +  2q^{26} + q^{27} + q^{28} - 2q^{29}
    + 2q^{30} - 5q^{32} + 4q^{33} + 6q^{34} + 2q^{35}- q^{36} + 6q^{37} - 4q^{38} - 2q^{39 }- 6q^{40} + 2q^{41} + q^{42} - 4q^{43} - 4q^{44} - 2q^{45} - q^{48} + q^{49} + q^{50}+
\cdots$
\end{center}

Clearly there are no companions of weight $2$ and level $3 $ or $7.$
Modulo $5^2$,  $f$ has no companion forms of weight $18$, level dividing $21$ and trivial character. Thus $f$ does not split mod ${5^2}.$
\end{example}

\begin{example}\label{level 57}
Let  $f$ be the newform of weight $4$, level $57$ and trivial character with Fourier expansion

\begin{center}
$f = q - q^2 + 3q^3 - 7q^4 - 12q^5 - 3q^6 - 20q^7 + 15q^8 + 9q^9 + 12q^{10} -
4q^{11} - 21q^{12} - 76q^{13} + 20q^{14} - 36q^{15} + 41q^{16} + 22q^{17} - 9q^{18} -
19q^{19} + 84q^{20} - 60q^{21} + 4q^{22} + 82q^{23} + 45q^{24} + 19q^{25} + 76q^{26} + 27q^{27} + 140q^{28} + 242q^{29} + 36q^{30} - 126q^{31} - 161q^{32} - 12q^{33} - 22q^{34} + 240q^{35} - 63q^{36} - 180q^{37} + 19q^{38} - 228q^{39} - 180q^{40} - 390q^{41} + 60q^{42} +
    308q^{43} + 28q^{44} - 108q^{45} - 82q^{46} - 522q^{47}+ 123q^{48} + 57q^{49} - 19q^{50} + \cdots$
    \end{center}
It has a unique   mod ${5}$ companion form $g$  of weight $2$, level $57$ and trivial character with Fourier expansion

\begin{center}
$g= q - 2q^2 - q^3 + 2q^4 - 3q^5 + 2q^6 - 5q^7 + q^9 + 6q^{10} + q^{11} - 2q^{12} +
2q^{13} + 10q^{14} + 3q^{15} - 4q^{16} - q^{17} - 2q^{18} - q^{19} - 6q^{20} + 5q^{21} -
2q^{22} - 4q^{23} + 4q^{25} - 4q^{26} - q^{27} - 10q^{28}
    - 2q^{29} - 6q^{30} - 6q^{31} + 8q^{32} - q^{33} + 2q^{34}+ 15q^{35} + 2q^{36} + 2q^{38} - 2q^{39} - 10q^{42 }- q^{43} + 2q^{44} - 3q^{45} + 8q^{46} - 9q^{47} + 4q^{48 }+ 18q^{49} - 8q^{50}+\cdots$
\end{center}
and no other companions of level dividing $57.$
Modulo $5^2$, $f$ has no companion forms  of weight $18$, level dividing $57$ and trivial character.

\end{example}

\section*{Acknowledgements}
The first author is grateful to  Paul Mezo and Yuly Billig of Carleton University,
Gabor Wiese of IEM, Essen and Universit\"{a}t Regensburg  for their postdoctoral
support that brought this work to fruition.
The authors thank the referee for  valuable comments and suggestions. Thanks also to Frazer Jarvis and Neil Dummigan for  useful discussions on this work and Panagiotis Tsaknias for his  assistance with the computations.

\bibliographystyle{abbrv}

\begin{thebibliography}{10}

\normalsize
\baselineskip=17pt



\bibitem{DDT}
H.~Darmon, F.~Diamond and R.~Taylor.
\newblock Fermat's Last Theorem.
\newblock In
{\em Current developments in mathematics,} 1995 Ed. R. Bott et al., International Press, Boston 1995



\bibitem{Em}
M.~Emerton.
\newblock A $p$-adic variational hodge conjecture and modular forms with
  complex multiplication.
\newblock {\em preprint}.

\bibitem{Flach}
M.~Flach.
\newblock A finiteness theorem for the symmetric square of an elliptic curve.
\newblock {\em Inventiones Mathematicae}, 109:307--327, 1992.
\newblock 10.1007/BF01232029.

\bibitem{Gee}
T.~Gee.
\newblock Companion forms over totally real fields. {II}.
\newblock {\em Duke Math. J.}, 136(2):275--284, 2007.

\bibitem{G}
E.~Ghate.
\newblock On the local behaviour of ordinary modular galois representations.
\newblock {\em Progress in Mathematics}, 224:105--224, 2004.

\bibitem{GV}
E.~Ghate and V.~Vatsal.
\newblock On the local behaviour of ordinary {$\Lambda$}-adic representations.
\newblock {\em Ann. Inst. Fourier (Grenoble)}, 54(7):2143--2162 (2005), 2004.

\bibitem{Gross}
B.~H. Gross.
\newblock A tameness criterion for {G}alois representations associated to
  modular forms (mod {$p$}).
\newblock {\em Duke Math. J.}, 61(2):445--517, 1990.

\bibitem{Hi}
H.~Hida.
\newblock CM components of the big Hecke algebra.
\newblock A series of lectures at Hokkaido University.



\bibitem{JM}
J.~Manoharmayum.
\newblock Lifting {G}alois representations of number fields.
\newblock {\em J. Number Theory}, 129(5):1178--1190, 2009.


\bibitem{mazur2}  B.~Mazur.
\newblock Deformation theory of Galois
representations.
\newblock In {\em Modular Forms and Galois Representations.} Eds G.~Cornell, J.~H.~Silverman,
G.~Stevens.  Springer, New York, 1997.

\bibitem{MW}
B.~Mazur and A.~Wiles.
\newblock On {$p$}-adic analytic families of {G}alois representations.
\newblock {\em Compositio Math.}, 59(2):231--264, 1986.

\bibitem{NSW}
J.~Neukirch, A.~Schmidt, and K.~Wingberg.
\newblock {\em Cohomology of number fields}, volume 323 of {\em Grundlehren der
  Mathematischen Wissenschaften [Fundamental Principles of Mathematical
  Sciences]}.
\newblock Springer-Verlag, Berlin, second edition, 2008.


\bibitem{Ram}
R.~Ramakrishna.
\newblock Deforming {G}alois representations and the conjectures of {S}erre and
  {F}ontaine-{M}azur.
\newblock {\em Ann. of Math. (2)}, 156(1):115--154, 2002.

\bibitem{Se}
J.-P. Serre.
\newblock {\em Abelian {$l$}-adic representations and elliptic curves},
  volume~7 of {\em Research Notes in Mathematics}.
\newblock A K Peters Ltd., Wellesley, MA, 1998.
\newblock With the collaboration of Willem Kuyk and John Labute, Revised
  reprint of the 1968 original.

\bibitem{SW1}
C.~M. Skinner and A.~J. Wiles.
\newblock Nearly ordinary deformations of irreducible residual representations.
\newblock {\em Ann. Fac. Sci. Toulouse Math. (6)}, 10(1):185--215, 2001.

\bibitem{Tay1}
R.~Taylor.
\newblock On icosahedral {A}rtin representations. {II}.
\newblock {\em Amer. J. Math.}, 125(3):549--566, 2003.


\bibitem{W1}
A.~Wiles.
\newblock On ordinary {$\lambda$}-adic representations associated to modular
  forms.
\newblock {\em Invent. Math.}, 94(3):529--573, 1988.


\end{thebibliography}

\end{document}